\documentclass[12pt,fleqn]{article}
\usepackage{amsfonts}
\usepackage{amssymb}
\usepackage{amscd}
\usepackage{mathrsfs}
\usepackage{amsmath,amssymb, amsthm}
\usepackage{graphicx}
\usepackage{indentfirst,latexsym,amssymb}
\usepackage[bookmarks=false]{hyperref}
\newtheorem{theorem}{Theorem}[section]
\newtheorem{pro}[theorem]{Proposition}
\newtheorem{lem}[theorem]{Lemma}
\newtheorem{rem}[theorem]{Remark}

\def\Aut{\hbox{\rm Aut}}

\def\GL{\hbox{\rm GL}}

\def\AGL{\hbox{\rm AGL}}

\def\G{\mathcal{G}}
\def\M{\mathcal{M}}

\def\V{\hbox{\rm V}}

\def\Cos{\hbox{\rm Cos}}
\def\Z{\mathbb{Z}}

\def\mod{\hbox{\rm mod }}

\renewcommand\arraystretch{0.8}

\def\a{\alpha}  \def\g{\gamma}  
 \def\s{\sigma}

\def\G{\Gamma}   \def\MM{{\cal M}} 
  
  \def\og{\overline G}

 \def\of{\overline F}

\def\ola{\overline a}\def\olc{\overline c}

\def\di{\bigm|} \def\lg{\langle} \def\rg{\rangle}
\def\nd{\mathrel{\bigm|\kern-.7em/}}

 \def\Exp{\hbox{\em Exp}}

\def\Aut{\hbox{\rm Aut}}

\def\mod{\hbox{\rm mod }}

\def\V {\hbox{\rm V}}

\def\AGL{\hbox{\rm AGL}}
\def\GL{\hbox{\rm GL}}

\def\Exp{\hbox{\rm Exp}}

\date{}
\begin{document}

\begin{center}
{\bf\Large  A Classification of  Orientably-Regular Embeddings \\of Complete Multipartite Graphs}
\end{center}

\begin{center}
{\sc Shaofei Du$^{a}$} and {\sc Jun-Yang Zhang$^{a,}$$^{b,}$}
\end{center}

\begin{center}
{\footnotesize  $^{a}$School of Mathematical Sciences, \\ Capital Normal University, Beijing 100048, P.R.China\\
$^{b}$ Department of Mathematics and Information Science, \\Zhangzhou Normal University,
 Zhangzhou, Fujian 363000, P.R.China}
\end{center}

\renewcommand{\thefootnote}{\empty}%{footnote}}
\footnotetext{Supported by NNSF(10971144).}

{\large
\begin{abstract}
Let $K_{m[n]}$ be the complete multipartite graph with $m$ parts, while each part contains $n$ vertices. The orientably-regular embeddings of complete graphs $K_{m[1]}$ have been determined by Biggs (1971)~\cite{Big1}, James~and~Jones (1985)~\cite{JJ}.  During the past twenty years, several papers such as Du et al.(2007, 2010)~\cite{DJKNS1,DJKNS2}, Jones et al. (2007, 2008)~\cite{JNS1,JNS2}, Kwak and Kwon (2005, 2008)~\cite{KK1,KK2} and
Nedela et al. (1997, 2002)\cite{NS,NSZ} contributed to the orientably-regular embeddings
of complete bipartite graphs $K_{2[n]}$ and the final classification was
given by Jones \cite{Jon1} in 2010. Based on  our former paper  \cite{ZD}, this paper gives a complete
classification of  orientably-regular embeddings of
graphs $K_{m[n]}$ for the general cases $m\ge 3$ and $n\ge 2$.
\end{abstract} }

%\vspace{3cm} {\bf Keywords}: regular embedding; complete multipartite graph; automorphism group;
%isobicyclic group.

%\vskip 1cm Corresponding author e-mail: dushf@mail.cnu.edu.cn

%\newpage
\section{Introduction}
A (topological) \emph{map} is a cellular decomposition of a closed surface.
A common way to describe such a map is to view it as a 2-cell embedding of
a connected graph or multigraph $\G$ into the surface $S$.  The components of the
complement $S \setminus \G$ are simply-connected regions called the \emph{faces}
of the map (or the embedding).
An \emph{automorphism} of a map $\MM $ is an automorphism of the underlying
(multi)graph $\G$ which extends to a self-homeomorphism of the supporting surface $S$.
It is well known that the automorphism group $\Aut(\MM)$ of a map $\MM$  acts semi-regularly on the set
of all incident vertex-edge-face triples (or \emph{flags} of $\G$). In particular, if $\Aut(\MM)$ acts regularly on the flags, we call it a {\em regular map}.
In the orientable case, if the group of all orientation-preserving
automorphisms of $\MM$ acts regularly on the set of all incident vertex-edge pairs (or \emph{arcs})
of $\MM$, then  we call $\MM$ an \emph{orientably regular\/} map.  Such maps fall
into two classes: those that admit also orientation-reversing automorphisms,
which are called \emph{reflexible}, and those that do not, which are  \emph{chiral}. Therefore, a reflexible map is a regular map but a chiral map is not.

One of the central problems in topological graph theory is to
classify all the regular embeddings in orientable or nonorientble
surfaces of a given graph. In a general setting, the
classification problem was treated by Gardiner,  Nedela, \v Sir\'
a\v n and \v Skoviera in \cite{GNSS}. However, for particular
classes of graphs, it has been solved only in a few cases.
Let $K_{m[n]}$ be the complete multipartite graph
with $m$ parts, while each part contains $n$ vertices.
All the  regular embeddings of complete graphs $K_{m[1]}$  have been determined
by Biggs, James and Jones \cite{Big1,JJ} for orientably case  and
by Wilson \cite{Wil} for nonorentably case.
As for the complete bipartite graphs $K_{2[n]}$,
the nonoritenably regular embeddings of these
graphs have recently been classified by Kwak and Kwon \cite{KK3};
during the past twenty
years,  several  papers \cite{DJKNS1,DJKNS2,JNS1,JNS2,KK1,KK2,NSZ} contributed to the orientably
  case, and  the  final classification was given by   Jones \cite{Jon1} in 2010.
Since then, the  classification  for general case $m\ge 3$ and
$n\ge 2$  has become an attractive topic in this research field.
The only known result is the determination of such   embeddings for $n=p$ a prime, given   by Du, Kwak and Nedela in \cite{DKN1}.

In this paper, we shall classify the orientably-regular embeddings of complete multipartite graphs. A start point is the main result in our former paper \cite{ZD}, namely the following reduction theorem.
\begin{pro}[reduction theorem]\label{rth}  Let $\M$ be an orientably-regular embedding of $K_{m[n]}$ where $m\ge 3$ and $n\ge 2$, with the group $\Aut^{+}(\M)$ of all orientation-preserving automorphisms.
 Let $\Aut^{+}_{0}(\M)$ be the normal subgroup  of $\Aut^{+}(\M)$ consisting of automorphisms preserving each part setwise.
Then $\Aut^{+}_{0}(\M)$ is an isobicyclic group. Moreover, we have
\begin{enumerate} \item[{\rm (1)}] if $m\ge 4$,  then $m=p$ and $n=p^e$ for some prime $p$; or
 \item[{\rm (2)}] if $m=3$, then $\Aut^{+}_{0}(\M)=Q\times K$, where $Q$ is a 3-subgroup (may be trivial) and $K$ is  an abelian $3'$-subgroup.
     \end{enumerate}
  \end{pro}

In Proposition~\ref{rth}, an isobicyclic group means a group
    $H=\lg x\rg \lg y\rg $, where $|x|=|y|=n$, $\lg x\rg \cap \lg y\rg =1$ and there exists an involution  $\a \in \Aut(H)$ such that $x^\a=y$.
    Throughout the paper, we call $(H,x,y)$ a  $n$-isobicyclic triple and   it plays an important role in the  classification of orientably-regular embeddings of $K_{m[n]}$.

As usual, the  orientably-regular map will be presented by a triple $(G; a, b)$ for a group $G$ generated by an element $a$ and an involution $b$ (see Section 2 for the details). The following is the main theorem of this paper.

\begin{theorem}[classification theorem]\label{main} For  $m\ge 3$ and $n\ge 2$, let $K_{m[n]}$ be the complete multipartite graph
with $m$ parts, while each part contains $n$ vertices. Suppose that $\MM$ is an orientably-regular embedding of $K_{m[n]}$ with  the group $G$ of all orientation-preserving automorphisms. Then $G$ and $\MM$  are given by
\begin{enumerate}
   \item $m=p\ge 5$, $n=p^e$ for a prime $p$:

   $G_{1}(p,e)=\langle a,c|a^{p^{e}(p-1)}=c^{p^{e+1}}=1, c^{a}=c^{r}\rangle$, where $r$ is a given generator of $\Z_{p^{e+1}}^*$;

 $\M_1(p,e,j)=\MM\big(G_{1}(p,e); a^{j}, a^{\frac{p^{e}(p-1)}{2}}c\big)$,
 where $j\in \Z_{p^{e}(p-1)}^{*}$.

 \item $m=n=p\ge 5$  for a prime $p$:

$\begin{array}{rl}
   G_{2}(p)=\hspace{-2mm}& \hspace{-2mm}\langle w,z\rangle\rtimes\langle c,g\rangle
   =\langle w,z,c,g\mid w^{p}=z^{p}=c^{p}=g^{p-1}=1, [w,z]=1,  c^{g}=c^{t}, \\
   & \hspace{48mm}w^c=wz, z^c=z, w^g=w,  z^{g}=z^{t}\rangle,
 \end{array}$

where $t$ is a given generator of $\Z_{p}^*$;

 $\M_2(p,j)=\M\big(G_{2}(p); wg^{j}, cg^{\frac{p-1}{2}}\big)$,
 where $j\in \Z_{p-1}^{*}$.

 \item  $m=p=3$, $n=k 3^e$ for $3\nmid k$ and $e\ne 0$:

$\begin{array}{rl}
     G_{3}(k,e)\hspace{-2mm}= & \hspace{-2mm}\langle a,b\mid a^{2\cdot3^{e}k}=b^{2}=1, c=a^{3^{e}}b, a^{2\cdot3^{e}}=x_{1}, x_{1}^{b}=y_{1}, [x_{1},y_{1}]=1, c^{3^{e+1}}=1, \\
      &\hspace{10mm}y_{1}^{a}=x_{1}^{-1}y_{1}^{-1}, c^{a}=c^{2}x_{1}^{u}y_{1}^{\frac{1-3^{e}}{2}u}\rangle,
   \end{array}$

where, if $k=1$, then $c^a=c^2$; if $k\ge 2$, then  $u3^{e}\equiv 1\pmod k$;

$\M_{3}(k,e,j)=\M(G_{3}(k,e); a^{j}, b)$, where $j\in \Z_{2\cdot 3^ek}^*$ and $\M_3(k,e,j_1)\cong \M_3(k,e,j_2)$ if and only if $j_1\equiv j_2\pmod{2\cdot 3^e}$.
\item  $m=p=3$, $n=3^ek$ for $3\nmid k$:

 $\begin{array}{rl}
   G_{4}(k,e,i,l)=\hspace{-2mm} & \hspace{-2mm}\langle a, b\mid a^{2n}=b^2=1, a^2=x, x^b=y, [x,y]=x^{\frac {in}3}y^{-\frac {in}3}, y^{a}=x^{-1}y^{-1},\\
    & \hspace{9mm}(ab)^3=x^{\frac {ln}3}y^{-\frac {ln}3}\rangle,
 \end{array}$

where, if $e=0$ then  $(i,l)=0$;  if $e=1$ then  $(i,l)=(0, 0)$, $(0,1)$; if $e\ge 2$ then  $(i,l)=(0,0),(0,1),(1,0),(1,1)$ or $(1,-1)$;

$\M_{4}(k,e, i, l,j)=\M(G_{4}(k,e,i,l); a^{j}, b)$, where $j=1$ for $(i,l)=(0,0)$  and $j=\pm1$ for other cases;

\end{enumerate}
The above maps are unique determined by the given parameters. The following two tables give  the enumerations for these maps.

{\renewcommand\arraystretch{1}\begin{center}
\begin{tabular}{l|l|c|l}
\multicolumn{4}{c}{\textbf{Table 1:} Enumerations of the resulting maps}\\ [5pt]\hline

Maps  & Number & Reflexible &   Type \\
   &   & or Chiral&    \\
 \hline
 $\M_{1}(p,e,j)$ &$p^{e-1}(p-1)\phi(p-1)$ & C  &$\{p^{e}(p-1),p^{e}(p-1)\}$,\\
    & &  & if~$p\equiv1~\pmod4$;\\
     & &  & $\{\frac{p^{e}(p-1)}{2},p^{e}(p-1)\}$,\\
    & &  & if $p\equiv3~\pmod4$.\\

  $\M_{2}(p,j)$, $p\ge 5$ &$\phi(p-1)$ & C  &$\{p(p-1),p(p-1)\}$,\\
    & &  & if $p\equiv1~\pmod4$;\\
      & &  & $\{\frac{p(p-1)}{2},p(p-1)\}$,\\
    & &  &if $p\equiv3~\pmod4$.\\
  $\M_{3}(k,e,j)$ & $2\cdot3^{e-1}$ & C &$\{3^{e+1}, 2n\}$\\
  $\M_{4}(k,e,0,0,1)$ & $1$ & R  &$\{3,2n\}$\\
  $\M_{4}(k,e,0,1,\pm1)$ & $2$ & C  &$\{9,2n\}$\\
  $\M_{4}(k,e,1,0,\pm1)$ & $2$ & C  &$\{3,2n\}$\\
  $\M_{4}(k,e,1,\pm1,\pm1)$ & $4$ & C  &$\{9,2n\}$\\
    \hline
\end{tabular}
\end{center}}
{\renewcommand\arraystretch{1}
\begin{center}
\begin{tabular}{ll|c|c|c}
\multicolumn{5}{c}{\textbf{Table 2:} Total numbers of regular embeddings of $K_{m[n]}$}\\[5pt]
\hline m & n &Reflexible &   Chiral &  Total \\ \hline
 $3$ & $k$ & $1$ & $0$ & $1$ \\
   & $3k$ & $1$ & $2$ & $3$ \\
   & $3^{e}k (e\ge 2)$ & $1$ &$2\cdot3^{e-1}+8$ & $2\cdot3^{e-1}+9$ \\
   $p\ge 5$ & $p$ & $0$ & $p\phi(p-1)$ & $p\phi(p-1)$\\
 & $p^{e}(e\ge 2)$ & $0$ & $p^{e-1}(p-1)\phi(p-1)$ & $p^{e-1}(p-1)\phi(p-1)$ \\
    \hline
\end{tabular}
\end{center}}
\end{theorem}

\begin{rem} \label{rem1} From the Theorem~\ref{main} and its  proof in Sections 3 and 4, we may obtain  some remarks as follows.
\begin{enumerate}
   \item $\M_1(p,e,j)$ and $\M_3(1, e, j)$  are   two families   of Cayley maps.
   \item Let  $P$ be a Sylow $m$-subgroup of $G=\Aut^{+}(\MM)$ ($m=p\geq5$ or $m=3$) and let $\Exp(P)$ be  the exponent of $P$. Then   $\Exp(P)=m^{e+1}$ for the groups $G_{1}(p,e)$ or $G_{3}(k,e)$; and $\Exp(P)=m^e$ for the groups $G_{2}(p)$ or $G_{4}(k,e,i,l)$.
   \item For $m=p\geq5$, $H=\Aut^{+}_{0}(\M)$ is a $p^{e}$-isobicyclic group and hence $H'$ must be a cyclic group of order at most $p^{e-1}$. It is interesting that either $H'\cong\Z_{p^{e-1}}$ (with the biggest possible order) or $H'=1$ (with the smallest possible order).

     Similarly, for $m=3$, $\Aut^{+}_{0}(\M)=Q\times K$ where $K$ is abelian, and   we have that either $Q'\cong \Z_{3^{e-1}}$ (with the biggest possible order), or $Q'\lesssim\Z_{3}$ (with the smallest  or the second smallest possible order).
\item The orientably-regular embeddings of $K_{m[p]}$ for prime $p$ have been  classified in \cite{DKN}, and they are precisely the maps $\M_1(p,1,j)$,  $\M_2(p,j)$, $\MM_3(1,1,j)$, $\M_4(1, 1, 0, 0, j)$ and $\M_4(p, 0, 0, 0, j)$.

The orientably-regular embeddings of $K_{3[n]}$ when $H=\Aut^{+}_{0}(\M)$ is abelian have been  classified in \cite{ZD}, and they are precisely the maps $\M_4(k,e, 0, l,j)$.

By examining the
Conder's lists of orientably-regular maps of type from 2 to 101 (see \cite{Conder1}), one can see that there are 15 orientably-regular embeddings of $K_{3[9]}$, which exactly coincides with our results here.
\end{enumerate}
\end{rem}

The paper is organized as follows. After this induction section, we describe the
orientably-regular maps in more details and give  some  preliminary results for later use in Section 2.  To classify the orientably-regular embeddings of $K_{m[n]}$ for the general cases $m\ge 3$ and $n\ge 2$, by Proposition \ref{rth} we only need to consider  the graphs  $K_{p[p^{e}]}$ for $p\ge 5$ a prime and the graphs $K_{3[n]}$ separately.
These two cases  will be dealt with  in Section 3 and  Section 4 respectively. Finally, the proof of  Theorem~\ref{main}  is  summarized in Section 5.

\section{Preliminaries}
Throughout this paper, all graphs are finite, simple and undirected. For a graph $\Gamma$, by $\mathrm{V}(\Gamma)$, $\mathrm{E}(\Gamma)$ and $\mathrm{D}(\Gamma)$ we  denote the vertex set, edge set and arc set of $\Gamma$, respectively. For any positive integer $n$,  set $[n]=\{1,\cdots,n\}.$ For a prime divisor $p$ of $n$, by $p^{d}\|n$ we denote that $p^{d}$ but $p^{d+1}\nmid n.$ For a ring $R$, we use $R^{*}$ to
denote the multiplicative group of $R$.
 The cyclic group of order $n$ as well as the integer residue ring modulo $n$ will be denoted by $\Z_{n}$ and
 the dihedral group of order $n$ will be denoted by $\mathbb{D}_{n}$.  By  $\mathrm{Z}(G)$,  $N_G(K)$ and $C_G(K)$
 we denote the center of the group $G$, the normalizer and the  centralizer of a subgroup $K$ in $G$, respectively.
  By $N\rtimes K$, we denote a semidirect product of $N$ by $K$ where $N$ is normal.
  For the notions not defined here, please refer \cite{Big2,Hup}.

It is well known that the automorphism group $G=\Aut(\M)$ of a regular map $\M$ is generated by
a generator $a$ of the stabilizer of a vertex $\gamma$ (which is necessarily cyclic) and an involution $b$
inverting an edge incident with $\gamma$, see \cite{GNSS}. Moreover, the embedding is determined by the group
$G$ and the choice of generators $a$ and $b$ \cite{NS,DKN}.

If the underlying graph is simple, then we may describe it by so called coset graphs.
Let $G=\lg a, b\rg $ be a  group where $a^t=b^2=1$ and $\lg a\rg $ is core-free. Let  $\Gamma=\Cos(G, \langle a\rangle, b)$  be the coset graph with vertex set $\V(\Gamma)=\{\langle a\rangle g\mid g\in G\}$ and arc set $\mathrm{D}(\Gamma)=\{(\langle a\rangle g,\langle a\rangle bg)\mid g\in G\}$. Then $G$ acts regularly on $\mathrm{D}(\Gamma)$ by right multiplication, the stabilizer of the vertex $\langle a\rangle1$ is the subgroup $\langle a\rangle$ of $G$ and $b$ is an involution inverting the arc $(\langle a\rangle1,\langle a\rangle b)$. By defining the local rotation $R$
by $(\lg a\rg g, \lg a\rg b g)^R=(\lg a\rg g, \lg a\rg b ag)$, we get an orientably-regular map, called    {\it algebraic map},
  denoted by $\M(G; a, b)$.

It is easy to show that two algebraic maps $\M(G; a, b)$ and $\M(G; a', b')$ are
isomorphic if and only if there is a group automorphism in $\Aut(G)$ taking $a\mapsto a'$ and $b\mapsto b'$. If the order of $ab$ and $a$ are  $s$ and $t$ respectively, then $\M(G;a,b)$ has type $\{s,t\}$ in the notation of Coxeter and Moser \cite{CM},
meaning that the faces are all $s$-gons and the vertices all have valency $t$. If $\M(G; a, b)$ and $\M(G; a^{-1}, b)$ are
isomorphic maps, then $\M(G; a, b)$ is reflexible, otherwise $\M(G; a, b)$ is chiral.

From  the above arguments, one can transfer the classification
problem of regular embeddings of a given graph into a purely group theoretical problem.
More precisely, one may classify all the regular maps with a given underlying graph $\Gamma$ of
valency $t$ in the following two steps:
\begin{enumerate}
  \item Find the representatives $G$ (as abstract groups) of the isomorphism classes of arc-regular subgroups of $\Aut(\Gamma)$ with cyclic vertex-stabilizers.
  \item For each group $G$ given in (1), determine all the algebraic regular maps $\M(G; a, b)$
with underlying graphs isomorphic to $\Gamma$, or equivalently, determine the representatives
of the orbits of $\Aut(\Gamma)$ on the set of generating pairs
$(a, b)$ of $G$ such that $|\langle a\rangle|=t$,
$|\langle b\rangle|=2$ and $\Cos(G,\langle a\rangle,b)\cong\Gamma$.
\end{enumerate}

\vskip 3mm
Now we give two lemmas for later use.

\begin{lem}\label{cos}
Suppose that $m$ is an odd prime and $n>2$ is an integer.  Let $G=\langle a,b\rg $  and $H=\langle x, y\rangle$ where  $a^{m-1}=x$ and $x^{b}=y$.
If $(H,x, y)$ is a $n$-isobicyclic triple, $H\unlhd G$, $G/H\cong \AGL(1,m)$ and $C_{G}(H)=\mathrm{Z}(H)$,  then $\M(G; a, b)$ is a regular embedding of $K_{m[n]}$.
\end{lem}
\begin{proof}
It suffices to show  that the coset graph  $\Gamma=\Cos(G,\langle a\rangle, b)$ is  a complete $m$-partite graph. Since $(H,x, y)$ is a $n$-isobicyclic triple, we have $H=\langle x\rangle\langle y\rangle$, $\langle x \rangle\bigcap\langle y \rangle=1$ and  $|H|=n^{2}$. Noting that $|G|=|G/H||H|=|\AGL(1,m)||H|=m(m-1)n^{2}$, we have $|G:\langle a\rangle|=mn$.
Since $a^{m-1}=x$, $(a^{b})^{m-1}=(a^{m-1})^{b}=y$, $H=\langle x,y\rangle$ and $C_{G}(H)=\mathrm{Z}(H)$, we have $\langle a\rangle\bigcap \langle a^{b}\rangle\leq C_G(x)\cap C_G(y)=C_G(H)=\mathrm{Z(}H)\le H$.
 It follows that $\langle a\rangle\bigcap \langle a^{b}\rangle\leq H\cap \langle a\rangle\bigcap\langle a^b\rangle
 =\langle x\rangle\bigcap\langle y\rangle
 =1$, that is $\langle a\rangle\bigcap \langle a^{b}\rangle=1$. Therefore $\Gamma$ is a simple graph of order $mn$ and valency $(m-1)n$.

Set $\Delta=\{\langle a\rangle h\mid h\in H\}$ and $\Sigma=\{\Delta g\mid g\in G\}.$
Then $\Sigma$ is a block system for $G$ acting on $\V(\Gamma)$.
Since $|\langle a\rangle H:\langle a\rangle|=|\langle a\rangle \langle x\rangle\langle y\rangle:\langle a\rangle|=|\langle a\rangle\langle y\rangle:\langle a\rangle|=n$, we have
$|\Delta|=n$ and then $|\Sigma|=m$. Clearly, $y^{i-j}\notin b\langle a\rangle$ for any two integers $i$ and $j$ in $[n]$.   Noting that $\Delta=\{\langle a\rangle h\mid h\in H\}=\{\langle a\rangle y^{i}\mid i\in[n]\}$,  $\Delta$ contains no pair of adjacent vertices. Therefore $\Gamma$ is a complete $m$-partite graph.\end{proof}

\begin{lem}\label{p4} Let $q=1+p^{f}$ where $p$ is a prime and $f\ge 1$ . If $p^{d}|k$ where $d\geq 1$, then
\begin{enumerate}
  \item[{\rm (1)}] $(q^{k}-1)/(q-1)\equiv k\pmod{p^{d+f}}$;
  \item[{\rm (2)}] $(q^{k+1}-1)/(q-1)\equiv k+1\pmod{p^{d+f}}$.
\end{enumerate}
\end{lem}
\begin{proof}
 (1)~~Since
$(q^{k}-1)/(q-1)=k+\sum\limits_{i=2}^{k}\binom{k}{i}p^{(i-1)f}$,
it suffices  to prove that $p^{d+f}|\binom{k}{i}p^{(i-1)f}$ for any  $2\le i\le k$.

The conclusion is clear  for  $i-1\geq \frac{d+f}{f}$ and so we
 assume that $i-1<\frac{d+f}{f}$. Then $2\le i<2+\frac{d}{f}\leq p^{d}$. Set $k=k'p^{d}$ and $i=i'p^{d_{i}}$ where $(p, i')=1$. Then $0\leq d_{i}<d$ and
\begin{equation*}
\binom{k}{i}
=\binom{k-1}{i-1}\frac ki=\binom{k-1}{i-1}\frac{k'}{i'}p^{d-d_{i}}.
\end{equation*}
Since $\binom{k}{i}$ is an integer and $(i',p)=1$, we have $i'\mid \binom{k-1}{i-1}k'$ and hence $\binom{k-1}{i-1}\frac{k'}{i'}$ is an integer as well. Noting that
\begin{equation*}
\binom{k}{i}p^{(i-1)f}
=\binom{k-1}{i-1}\frac{k'}{i'}p^{(i-1)f+d-d_{i}},
\end{equation*}
we have $p^{(i-1)f+d-d_{i}}\mid \binom{k}{i}p^{(i-1)f}$.
Noting $i'p^{d_i}=i\ge 2$, one may  check that $i'p^{d_{i}}-d_{i}-1\ge 1$. Then
\begin{equation*}
(i-1)f+d-d_{i}=(i'p^{d_{i}}-1)f+d-d_{i}\geq (i'p^{d_{i}}-d_{i}-1)f+d\geq f+d,
\end{equation*}
which implies   $p^{d+f}|\binom{k}{i}p^{(i-1)f}$.
\vskip4mm
(2) By the Binomial Theorem, we have
\begin{equation*}
q^{k}-1=(1+p^{f})^{k}-1=\binom{k}{1}p^{f}+\binom{k}{2}p^{2f}+\cdots+\binom{k}{k}p^{kf}.
\end{equation*}
Then we get $p^{d+f}\mid q^{k}-1$ since $p^{d}\mid k$. It follows that
$q^{k}\equiv 1\pmod{p^{d+f}}$. Since it has been proved that
\begin{equation*}
(q^{k}-1)/(q-1)\equiv k\pmod{p^{d+f}},
\end{equation*}
we get
\begin{equation*}
(q^{k+1}-1)/(q-1)=(q^{k}-1)/(q-1)+q^{k}\equiv k+1\pmod{p^{d+f}}.
\end{equation*}
\end{proof}

\section{Regular embeddings of $K_{p[p^{e}]}$}
In this section, we manly consider the case  $m=p$ and $n=p^{e}$ where $p\ge 5$ is a prime.  However, to obtain  some results  used  in Section 4,  we allow $p=3$ in  this section if no state explicitly.

Set $\G=K_{p[p^e]}$, with  the vertex set
\begin{equation*}
V(\G)=\bigcup_{i=1}^{p}\Delta_{i},~\mbox{where}~\Delta_{i}=
\{\gamma_{i1},\gamma_{i2},\cdots,\gamma_{i{p^e}}\}
\end{equation*}
and the edges  are all pairs $\{\gamma_{ij},\gamma_{kl}\}$ of vertices with $i\ne k$.
Then $\Aut(\G )=S_{p^e}\wr S_{p}$, which has blocks $\Delta_i$ where $1\le i\le p$.

Let $\MM$ be an orientably-regular map with the underlying graph  $\G$ and set $G=\Aut^{+}(\MM)=\lg a, b\rg $, where $\lg a\rg =G_{\gamma_{11}}$ and $b$ reverses  the arc $(\gamma_{11},\gamma_{21})$. We use $H=\Aut_{0}^{+}(\MM)$ to denote
the normal subgroup of $G$ consisting of automorphisms preserving each part setwise.
By Proposition~\ref{rth}, $H$ is a $p^e$-isobicyclic group.

By a result of Hupert~\cite{H}, $H$ is metacyclic.  One can see from \cite{JNS1} that
 \begin{equation}\label{H}
H=\langle x,z|x^{p^{e}}=z^{p^{e}}=1,z^{x}=z^{q}\rangle,
\end{equation}
where  $q=1+p^{f}$ for $f\in [e]$ and different $f$ give nonisomorphic groups. Particularly, $H\cong\Z_{p^{e}}\times\Z_{p^{e}}$ if $f=e$ and
$H$ is nonabelian if $f\in[e-1]$. Each element of $H$ can be written uniquely in the form $x^{i}z^{j}$ where $i,j\in \Z_{p^{e}}$, with the rules
\begin{equation}\label{M}
 (x^{i}z^{j})(x^{k}z^{l})=x^{i+k}z^{jq^{k}+l}\quad\mbox{and}\quad (x^{i}z^{j})^{k}=x^{ik}z^{j(q^{ik}-1)/(q^{i}-1)}.
\end{equation}
The center, derived subgroup and Frattini subgroup of $H$ are $\mathrm{Z}(H)=\langle x^{p^{e-f}},z^{p^{e-f}}\rangle$, $H'=\langle z^{p^f} \rangle$ and $\Phi(H)=\langle h^{p}\mid h\in H\rangle$, respectively. The exponent $\Exp(H)$ of $H$ is $p^{e}$.  Moreover, $H$ is a regular $p$-group, that is, all elements $h_{1}, h_{2}\in H$ satisfy
$(h_{1}h_{2})^{p}=h_{1}^{p}h_{2}^{p}c_{1}^{p}\cdots c_{k}^{p}$ where $c_{1},\cdots, c_{k}\in\langle h_{1},h_{2}\rangle'$.

Since $H_{\gamma_{11}}=\langle a^{m-1}\rangle$, one can set $a^{m-1}=x^iz^j$ where $p\nmid i$.   By Eq.(\ref{M}), we have $a^{i^{-1}(m-1)}=xz^{j'}$ for some $j'\in\Z_{p^{e}}$ and then
$z^{a^{i^{-1}(m-1)}}=z^x=z^{1+q}$. Replacing
$a$ and $x$ by $a^{i^{-1}}$ and $xz^{j'}$ respectively, we may assume that $a^{m-1}=x$.
Then $(H,x,y)$ is a $p^{e}$-isobicyclic triple by setting $y=x^{b}$.

Let $P$ be a Sylow $p$-subgroup of $G$. Then $P$ is an extension of $H$ by $\Z_p$.  Since $\Exp(H)=p^{e}$, we have $\Exp(P)=p^{e+1}$ or $p^{e}$ and then  we shall discuss these two cases in the following two subsections, separately.
\subsection{$\Exp(P)=p^{e+1}$}
\begin{theorem}\label{m1}
Suppose that $\Exp(P)=p^{e+1}$ where $p\ge 5$. Then
$\MM$  is isomorphic to one of the maps
 $\M_1(p,e,j)$ where $j\in \Z_{p^e(p-1)}^*$.
Moreover, all of the maps $\M_1(p,e,j)$ are orientably-regular embeddings of $K_{p[p^{e}]}$ and such maps are uniquely determined by the parameter $j$.
\end{theorem}
\begin{proof}  The proof is divided into two steps.

\vskip 3mm (1) Determination of the group $G$.

 \vskip 3mm
 Recalling that for $p\ge 5$,
\begin{equation*}
 G_{1}(p,e)=\langle a,c|a^{p^{e}(p-1)}=c^{p^{e+1}}=1, c^{a}=c^{r}\rangle,
\end{equation*}
 where $r$ is a given generator of $\Z_{p^{e+1}}^*$. If we allow $p=3$ for the groups $G_1(p,e)$ and maps $\M_1(p,e,j)$, then $G_1(3,e)$ and $\M_1(3,e,j)$ are exactly $G_3(1,e)$ and $\M_3(1,e,j)$ by choosing $r=2$ respectively. Therefore,  we allow $p=3$ in the following arguments.

Since  $\Exp(P)=p^{e+1}$, there exists an element $g$ of order $p^{e+1}$ in $G\setminus H$.  Clearly $\langle g\rangle$ permutates  the $p$ parts of $\Gamma$ and hence is regular on $V(\G)$.  Since $G_{\gamma_{11}}=\langle a\rangle$, we get  $\langle a\rangle\bigcap\langle g\rangle=1$ and $G=\langle a\rangle\langle g\rangle$, a product of two cyclic groups.
Then $G'$ is abelian  by an Ito's theorem in \cite{Ito}.  Thus $G'$ acts semiregularly on $V(\G)$, from which we have $G'\cap \lg a\rg =G\bigcap G_{\g_{11}}=1$. Furthermore, by \cite[Corollary C]{Conder} we know that $G'/(G'\cap \lg a\rg)$ is  isomorphic to a subgroup of $\lg b\rg $, which implies
 that $G'$ is cyclic.
Set $c=a^{\frac{p^{e}(p-1)}{2}}b$. Then $c^2=a^{\frac{p^{e}(p-1)}{2}}ba^{\frac{p^{e}(p-1)}{2}}b=[a^{\frac{p^{e}(p-1)}{2}}, b]\in G'$ and hence $c\in G'$. Since $G'$ is cyclic, $\lg c\rg \lhd G$ and thus $\lg c\rg \lg a\rg \le G$. From
 $G=\langle a,b\rangle=\langle a,c\rangle$, we get $G=\lg c\rg \lg a\rg$. In particular, $G'=\lg c\rg .$

\vskip 2mm
Set  $c^{a}=c^{r}$. Since $\lg a\rg $ is core-free, $c^{r^i}=c^{a^i}\ne c$ for any $i\not\equiv  0(\mod p^{e}(p-1))$. Therefore, $\Z_{p^{e+1}}^{*}=\langle r\rangle$.  Take two such $r$ and $r'$ and denote the corresponding groups by $G(r)$ and $G(r')$.
Set $r=r'^s$ for some integer $s$.  Then the mapping $\s :a\to a^s$, $c\to c$ gives an isomorphism from $G(r)$ to $G(r')$. Therefore, $r$ can be chosen to be any given generator  of $\Z_{p^{e+1}}^*.$

Now $G$ satisfies all the defining relations of  $G_{1}(p,e)$ (take $r=2$ if $p=3$). A direct checking shows that  $|G_{1}(p,e)|=|G|$ and so $G\cong G_{1}(p,e)$.

\vskip 2mm
(2) Determination  the map $\M$.
 \vskip 2mm
By the above proof, we know that $\MM$  is isomorphic to one of the maps
\begin{equation*}
\M_1(p,e,j)=\M\big(G_1(p,e);a^{j},b\big)~\mbox{where}~j\in \Z_{p^e(p-1)}^*.
\end{equation*}
By Lemma~\ref{cos}, all of the maps $\M_1(p,e,j)$ are regular embeddings of $K_{p[p^{e}]}$.  Suppose that for  two parameters $j_{1}$ and $j_{2}$, $\M_1(p,e,j_{1})\cong\M_1(p,e,j_{2})$. Then there exists an automorphism $\sigma$ of $G_1(p,e)$ such that $\sigma(a^{j_{1}})=a^{j_{2}}$ and $\sigma(b)=b$. If follows that
\begin{equation*}
\s (c)=\s (a^{\frac{p^{e}(p-1)}{2}}b)=\s(a^{j_1\frac{p^{e}(p-1)}{2}}b)=a^{j_{2}\frac{p^{e}(p-1)}{2}}b=
a^{\frac{p^{e}(p-1)}{2}}b=c,
\end{equation*}
 and hence
\begin{equation*}
c^{r^{j_{2}}}=c^{a^{j_{2}}}=\big(\sigma(c)\big)^{\sigma(a^{j_{1}})}
=\sigma(c^{a^{j_{1}}})=\sigma(c^{r^{j_{1}}})=c^{r^{j_{1}}}.
\end{equation*}
Therefore $r^{j_{1}}\equiv r^{j_{2}}\pmod{p^{e+1}}$ and then
$j_{1}\equiv j_{2}\pmod{p^{e}(p-1)}$.
Thus the maps $\M_1(p,e,j)$ are uniquely determined by the parameter $j$.
\end{proof}

\begin{rem}\label{r}
In Theorem~\ref{m1}, $H=\langle x, w\rg , $ where $x=a^{p-1}$ and $w=c^{p}$. Now
\begin{equation*}
H=\langle x,w\mid x^{p^{e}}=w^{p^{e}}=1, w^x=w^{r^{p-1}}\rangle.
\end{equation*}
Since $\Z_{p^{e+1}}^{*}=\langle r\rangle$, $r^{p-1}$ is of order $p^e$ in $\Z_{p^{e+1}}$.
It is well known that the subgroup of order $p^e$ of $\Z_{p^{e+1}}^{*}$ is $\{ 1+pk\di  0\le k\le p^{e}-1\}$.
Therefore, $p\di\di (r^{p-1}-1).$ Hence,  if $e\ge 2$, then $H$ is nonableian and $f=1$.
\end{rem}
\begin{lem}
 Let $p\geq5$. Then $\M_1(p,e,j)$ are chiral maps of type $\{p^{e}(p-1), p^{e}(p-1)\}$ if $p\equiv 1\pmod 4$; and $\{\frac{p^{e}(p-1)}{2},p^{e}(p-1)\}$ if $p\equiv 3 \pmod 4$.
\end{lem}
\begin{proof}
Since $j\not\equiv-j\pmod{p^{e}(p-1)}$ for all $j\in\Z_{p^{e}(p-1)}^{*}$, we have $\M_1(p,e,j)$ is not isomorphic to $\M_1(p,e,-j)$ and hence $\M_1(p,e,j)$ are chiral maps.

Set $l=j-\frac{p^e(p-1)}2$. Since $p\geq 5$, we get $p\nmid r^l-1$. Then for any integer $i$,
\begin{equation*}
(a^jb)^i=(a^lc)^i=a^{li}c^{1+r^{l}+\cdots+r^{(i-1)l}}=a^{li}c^{\frac{r^{li}-1}{r^l-1}}.
\end{equation*}
Noting that $l$ is odd
if $p\equiv 1\pmod 4$ and even if $p\equiv 3\pmod 4$, we have that the order of $a^jb$ is $p^e(p-1)$  if $p\equiv 1\pmod 4$ and $\frac{p^e(p-1)}2$ if $p\equiv 3\pmod 4$. It follows that the maps $\M_1(p,e,j)$ has the type $\{p^{e}(p-1), p^{e}(p-1)\}$ if $p\equiv 1\pmod 4$ and $\{\frac{p^{e}(p-1)}{2},p^{e}(p-1)\}$ if $p\equiv 3\pmod 4$.
\end{proof}
\subsection{$\Exp(P)=p^{e}$}
In this subsection, we discuss the case $\Exp(P)=p^{e}$. Before giving the main results, we prove two lemmas for later use.
\begin{lem}
\label{p5} $N:=\langle x^{p^{e-f}},z\rangle \unlhd G$.
\end{lem}
\begin{proof}  Set $N=\lg x^{p^{e-f}}, z\rg$. Since $H'=\langle z^{p^{f}}\rangle\le N$ and $H/H'=\lg x, z\rg/H'\cong \Z_{p^e}\times \Z_{p^{f}}$,
we have $N/H'=\{ gH'\in H/H'\mid (gH')^{p^f}=H'\}$, which is  characteristic
in $H/H'$ and hence is normal in $G/H'$. It follows that $N\unlhd G$.
\end{proof}

\begin{lem}\label{p6} Suppose that $x^{g}=x^{i}z^{j}$ and  $z^{g}=x^{k}z^{l}$ for some $g\in P\setminus H$. Then
\begin{enumerate}
  \item  $p^{e-f}\mid k$, $p^{e-f}|(i-1)$, $p\nmid j$ and $l\equiv 1\pmod p$;
  \item  one can set $x^{g}=xz$ by reselecting $z$.
\end{enumerate}

\end{lem}

\begin{proof} (1) By lemma \ref{p5}, we have $p^{e-f}\mid k$ and hence $p\nmid l$. It follows that $x^{kq}=x^{k}$ and $x^{k}\in Z(H)$. Then
\begin{equation*}
\left\{
  \begin{array}{l}
    (z^{g})^{q}=(x^{k}z^{l})^{q}=x^{kq}z^{lq}=x^{k}z^{lq}; \\
    (z^{q})^{g}=(z^{x})^{g}=(z^{g})^{x^{g}}
=(x^{k}z^{l})^{x^{i}z^{j}}=x^{k}(z^{l})^{x^{i}}
=x^{k}z^{lq^{i}}.
  \end{array}
\right.
\end{equation*}
Therefore we have $x^{k}z^{lq}=x^{k}z^{lq^{i}}$ and then $lq\equiv lq^{i}\pmod {p^{e}}$. Noting that $q=1+p^{f}$ and $p\nmid l$, we have $p\nmid lq$ and hence $q^{i-1}\equiv 1\pmod {p^{e}}$.
By~\cite[Lemma 6]{JNS1}, we get $p^{e-f}|(i-1)$.

Write $\overline{G}=G/\Phi(H)$. Then  $\overline{x}^{\overline{g}}=\overline{x}^{i-1}\overline{x}\overline{z}^{j}
=\overline{x}\overline{z}^{j}$,  $\overline{z}^{\overline{g}}=\overline{x}^{k}\overline{z}^{l}=\overline{z}^{l}$.
Since $x$ fixes the vertex $\gamma_{11}$ and $g$ moves  away any part of $\Gamma$,  we have $\overline{x}^{\overline{g}}\neq \overline{x}$. It follows that $\overline{g}$ can be represented on $\overline{H}\cong \Z_p^2$ as a matrix
$\left(\begin{array}{lr}
1 &  0\\
 j & l \\
\end{array}
\right)\in \GL(2,p)$ with respect to the basis $\{\overline{x},\overline{z}\}$. Since the order of $\overline{g}$ is $p$, we have $l\equiv 1\pmod p$ and $p\nmid j$.

\vspace{3mm}

(2)  Since $x^g=x (x^{i-1}z^j)$ and
$(x^{i-1}z^{j})^{x}=x^{i-1}(z^{j})^{x}=x^{i-1}z^{jq}=(x^{i-1}z^{j})^{q}$,
one may get the desired conclusion by replacing $z$ by $x^{i-1}z^{j}$.
\end{proof}

\begin{theorem}\label{ml} Suppose that $H$ is non-abelian and $\Exp(P)=p^{e}$. Then
 $p=3$ and $f=e-1$.
\end{theorem}
\begin{proof} Set $c=a^{\frac{p^e(p-1)}2}b$, $x^{c^{k}}=x^{u_{k}}z^{v_{k}}$ and $z^{c^{k}}=x^{s_{k}}z^{t_{k}}$ for $k\geq 1$.
By lemma \ref{p6}, one can set $u_{1}=v_{1}=1$,  $s_{1}=s$ and $t_{1}=t$
where  $p^{e-f}|s$ and $t\equiv1\pmod{p}$. Particularly, $x^s\in \mathrm{Z}(H)$. Now we prove the theorem by the following three steps.

\vskip 3mm
(1) Show that
\begin{equation}\label{e1}
{\footnotesize \left(
  \begin{array}{cc}
    u_{k} & s_{k} \\
    v_{k} & t_{k} \\
  \end{array}
\right)\equiv
\left(
  \begin{array}{cc}
    1 & s\\
    1 & t \\
  \end{array}
\right)^{k} }\pmod{p^{e}}.
\end{equation}

\vskip 2mm
We proceed the proof by induction on $k$.  The assertion is trivial if $k=1$. Let $k\ge 2$.
Then we have
\begin{equation*}
(xz)^{u_{k-1}}=x^{u_{k-1}}z^{(q^{u_{k-1}}-1)/(q-1)}\quad\mbox{and}\quad (x^{s}z^{t})^{v_{k-1}}=x^{sv_{k-1}}z^{tv_{k-1}}.
\end{equation*}
By Lemma~\ref{p6}., we have that $p^{e-f}$ divide both $u_{k-1}-1$ and $s_{k-1}$.
Then by Lemma~\ref{p4}., we get
\begin{equation*}
(q^{u_{k-1}}-1)/(q-1)\equiv u_{k-1}\pmod{p^{e}}\quad\mbox{and}\quad
(q^{s_{k-1}}-1)/(q-1)\equiv s_{k-1}\pmod{p^{e}}.
\end{equation*}
Therefore,
\begin{equation*}
\begin{array}{lll}
  x^{c^{k}}\hspace{-2mm}&=&\hspace{-2mm}(x^{u_{k-1}}z^{v_{k-1}})^{c}=(x^{c})^{u_{k-1}}(z^{c})^{v_{k-1}}
  = (xz)^{u_{k-1}}(x^{s}z^{t})^{v_{k-1}} \\
   &=&\hspace{-2mm}x^{u_{k-1}}z^{(q^{u_{k-1}}-1)/(q-1)}
   x^{sv_{k-1}}z^{tv_{k-1}}=
  x^{u_{k-1}+sv_{k-1}}z^{(q^{u_{k-1}}-1)/(q-1)
+tv_{k-1}}\\
&=&\hspace{-2mm}x^{u_{k-1}+sv_{k-1}}z^{u_{k-1}+tv_{k-1}}. \end{array}
\end{equation*}
 Similarly, we get $z^{c^{k}}=x^{s_{k-1}+st_{k-1}}z^{s_{k-1}+tt_{k-1}}$.
It follows that
\begin{equation*}
\left(\begin{array}{cc} u_{k}&s_{k}  \\v_{k} &t_{k} \\ \end{array}\right)\equiv \left(\begin{array}{cc} 1&s  \\ 1&t \\ \end{array}\right)\left(\begin{array}{cc} u_{k-1}&s_{k-1}  \\ v_{k-1}&t_{k-1} \\ \end{array}\right)\pmod{p^{e}}.
\end{equation*}
Then we get the conclusion by employing the inductive hypothesis.

\vskip 3mm
(2) Show that $p\|v_{p}$ if $p\geq 5$.

Suppose that $p\ge 5$. Set $A=\left(\begin{array}{cc} 0& s \\1&t-1\\ \end{array}\right)$. Noting that $p$ divide both $s$ and $t-1$. We have
\begin{equation*}
A^2\equiv 0\pmod p\quad\quad\mbox{and}\quad\quad A^4\equiv 0\pmod{p^2}.
\end{equation*}
By Eq.(\ref{e1}), we have
\begin{equation*}
\left(
  \begin{array}{cc}
    u_{p} & s_{p} \\
    v_{p} & t_{p} \\
  \end{array}
\right)=(E+A)^p\equiv E+pA=\left(
  \begin{array}{cc}
    1 & ps \\
    p &1+p(t-1)\\
  \end{array} \right)\pmod{p^2},
\end{equation*}
from which we have $p\|v_{p}$.

\vskip 3mm
(3) Show that $p=3$ and $f=e-1.$

\vskip 2mm
Set $c^{p}=x^{i}z^{j}$. Since $\Exp(P)=p^e$, we have that $p$ divides both $i$ and $j$.
Then we have
\begin{equation*}
x^{c^{p}}=x^{x^{i}z^{j}}=x^{z^{j}}=xz^{-jq}z^{j}=xz^{-jp^{f}}\quad\mbox{and}\quad z^{c^{p}}=z^{x^{i}z^{j}}=z^{x^{i}}=z^{q^{i}}.
\end{equation*}
It follows that
\begin{equation}\label{4}
u_p\equiv  1,~v_p\equiv  -jp^{f},~s_p\equiv  0,~t_p\equiv  q^{i}\pmod{p^{e}}.
\end{equation}
Since $p\mid j$, we have $p^{2}|v_{p}$. Combining with the result of \emph{Step 2}, we get
$p=3$. It can be straightforward to calculate that
 \begin{equation}\label{5}
\left(\begin{array}{cc}
 1 &s  \\
   1 &t \\
   \end{array}
   \right)^{3}
\equiv
\left(
\begin{array}{cc}
 1+s(t+2) &  s(1+t+t^2+s)\\
   1+t+t^2+s &s+2st+t^3 \\
    \end{array}
     \right)\pmod{3^{e}}.
     \end{equation}
By Eq.(\ref{4}) and Eq.(\ref{5}), we have
\begin{equation*}
    \left(
\begin{array}{cc}
 1+s(t+2) &  s(1+t+t^2+s)\\
   1+t+t^2+s &s+2st+t^3 \\
    \end{array}
     \right)\equiv
      \left(
\begin{array}{cc}
 1 &  0\\
   -3^fj&(1+3^{f})^i \\
    \end{array}
     \right)
     \pmod{3^{e}}.
\end{equation*}
Particularly,
\begin{equation}\label{eq3}
    1+t+t^2+s\equiv-3^fj\pmod{3^{e}}
\end{equation}
Since $3\mid (t-1)$ and $3\mid j$, we have $3\|(1+t+t^{2})$ and $9\mid (-3^fj)$. Then by Eq.(\ref{eq3}), we get $3\|s$. Recalling that $3^{e-f}|s$, we have $f=e-1$.
\end{proof}

\begin{theorem} \label{abelian} Suppose that $p\ge 5$, $H$ is abelian and $\Exp(P)=p^{e}$. Then $\MM$ is isomorphic to one of the maps
$\MM_2(p,j)$ which are uniquely determined by the parameter $j\in \Z_{p-1}^*$.  Moreover, all of the maps $\MM_2(p,j)$ are orientable  regular embeddings of $K_{p[p]}$ and they are chiral maps with the type $\{p(p-1), p(p-1)\}$ if $p\equiv 1\pmod 4$ and $\{\frac{p(p-1)}{2},p(p-1)\}$ if $p\equiv 3 \pmod 4$.
\end{theorem}
\begin{proof}  Suppose that $p\ge 5$, $H$ is abelian and $\Exp(P)=p^{e}$.
By the classification of orientably-regular embeddings of $K_{p[p]}$ in \cite{DKN}, it suffices to show that $e=1$.

Recalling the notations:  $\lg a\rg =G_{\gamma_{11}}$, $b$ reverses  the arc $(\gamma_{11},\gamma_{21})$, $x=a^{p-1}$, $y=x^{b}$ and $H=\langle x,y\rangle$.  Considering the conjugacy action of $a$ on $H$, we have $x^{a}=x$ and set $y^{a}=x^{s}y^{t}$ for integers $s$ and $t$. Then
\begin{equation*}
y=y^{x}=y^{a^{p-1}}=x^{s(1+t+\cdots+t^{p-2})}y^{t^{p-1}},
\end{equation*}
and hence
\begin{equation*}
s(1+t+\cdots+t^{p-2})\equiv 0\pmod{p^{e}}~~~\mbox{and}~~~t^{p-1}\equiv 1\pmod{p^{e}}.
\end{equation*}
For  any $i\in [p-2]$,  noting that $a^i$ moves away the block $\Delta_2$ and $H_{\gamma_{21}}=\langle(a^{p-1})^b\rangle=\langle y\rangle$, we have $\langle y^{a^{i}}\rangle\bigcap\langle y\rangle=1$. Therefore,
$t$ is of  order $p-1$ modulo $p^{e}$.
It follows that  $t^{\frac{p-1}{2}}\equiv-1\pmod{p^{e}}$ and then
\begin{equation*}
(t-1)\big(1+t+\cdots+t^{\frac{p^{e}(p-1)}{2}-1}\big)=
t^{\frac{p^{e}(p-1)}{2}}-1\equiv-2\pmod{p^{e}}.
\end{equation*}
Therefore $p\nmid (t-1)$ and hence
\begin{equation*}
1+t+\cdots+t^{\frac{p^{e}(p-1)}{2}-1}\equiv\frac{2}{1-t}\pmod{p^{e}}.
\end{equation*}
Let $c=a^{\frac{p^{e}(p-1)}{2}}b$ and $v=\frac{2s}{1-t}$. Then
\begin{equation}\label{action}
x^{c}=y~\mbox{and}~y^{c}=\Big(x^{s\big(1+t+\cdots+t^{\frac{p^{e}(p-1)}{2}-1}\big)}
y^{t^{\frac{p^{e}(p-1)}{2}}}\Big)^{b}
=(x^{v}y^{-1})^{b}=x^{-1}y^{v}.
\end{equation}
The conjugacy action of $G$ on $H$ gives a unfaithful homomorphism $\pi $ from $G$ to $\Aut(H)$. By Eq.(\ref{action}),  $\pi(c)$ is represented  as a matrix
\begin{equation*}
\mathrm{C}= \left(
              \begin{array}{cc}
                0 & -1 \\
                1 & v \\
              \end{array}
            \right)
\in\GL(2,\Z_{p^{e}})
\end{equation*}
with respect to two generators  $x$ and $y$.
Set $\overline{G}=G/H$. Then $\overline{G}=\lg \ola, \olc \rg \cong \AGL(1,p)$. Since the product of two different involutions in $\AGL(1,p)$ must be of order $p$, the order of $\overline{c}$ is $p$ and hence $c^{p}\in H$.
Therefore we have
\begin{equation}\label{identity}
 \mathrm{C}^p=\left(
             \begin{array}{cc}
               1 & 0 \\
               0 & 1 \\
             \end{array}
           \right)
\pmod{p^e}.
\end{equation}
It follows that
 $\mathrm{C}^p=\left(
             \begin{array}{cc}
               1 & 0 \\
               0 & 1 \\
             \end{array}
           \right)
\pmod{p}$.
Noting that any matrix of order $p$ in $\GL(2,p)$ has the eigenvalue 1, one gets
$v\equiv 2\pmod p$.
Set
\begin{equation*}
v-2=rp,~
\mathrm{A}=\left(
              \begin{array}{cc}
                0 & -1 \\
                1 & 2 \\
              \end{array}
            \right)~\mbox{and}~\mathrm{B}=\left(
              \begin{array}{cc}
                0 & 0 \\
                0 & 1 \\
              \end{array}
            \right).
\end{equation*}
Then
\begin{equation*}
 \mathrm{C}^p=\big(\mathrm{A}+rp\mathrm{B}\big)^p\equiv \mathrm{A}^p+rp(\mathrm{A}^{p-1}\mathrm{B}+\mathrm{A}^{p-2}\mathrm{B}\mathrm{A}
 +\cdots+\mathrm{B}\mathrm{A}^{p-1})\pmod{p^{2}}.
\end{equation*}
It can be straightforward to check that
$\mathrm{A}^{k}=\left(
              \begin{array}{cc}
                1-k & -k \\
                k & k+1 \\
              \end{array}\right)$ for all $k\geq1$. Therefore
\begin{eqnarray*}
   &&\hspace{-6mm}  \\
   \mathrm{C}^{p}\hspace{-2mm}&\equiv&\hspace{-2mm}\mathrm{A}^p+
   rp\sum_{i+j=p-1}\mathrm{A}^{i}\mathrm{B}\mathrm{A}^{j}\\
      &=&\hspace{-2mm}\left(
        \begin{array}{cc}
          1-p & -p \\
          p & p+1 \\
        \end{array}
      \right)+
      rp\sum_{i+j=p-1}\left(
        \begin{array}{cc}
          1-i & -i \\
          i & i+1 \\
        \end{array}
      \right)
      \left(
        \begin{array}{cc}
          0 & 0 \\
          0 & 1 \\
        \end{array}
      \right)
      \left(
        \begin{array}{cc}
          1-j & -j \\
          j & j+1 \\
        \end{array}
      \right)\\
   &=&\hspace{-2mm}\left(
        \begin{array}{cc}
          1-p & -p \\
          p & p+1 \\
        \end{array}
      \right)+
      rp\sum_{i=0}^{p-1}\left(
        \begin{array}{cc}
          i^{2}-(p-1)i & i^{2}-pi \\
          -i^{2}+(p-2)i+p-1 & -i^{2}+(p-1)i+p \\
        \end{array}
      \right)\\
   &\equiv&\hspace{-2mm}\left(
        \begin{array}{cc}
          1-p & -p \\
          p & p+1 \\
        \end{array}
      \right)+
      rp\sum_{i=0}^{p-1}\left(
        \begin{array}{cc}
          i^{2}+i & i^{2} \\
          -i^{2}-2i-1 & -i^{2}-i \\
        \end{array}
      \right)\\
     &=&\hspace{-2mm}\left(
                       \begin{array}{cc}
                         1-p & -p \\
                         p & p+1 \\
                       \end{array}
                     \right)+rp
                     \left(
                       \begin{array}{cc}
                         \frac{p(p-1)(2p-1)}{6}+\frac{p(p-1)}{2}& \frac{p(p-1)(2p-1)}{6} \\
                         -\frac{p(p-1)(2p-1)}{6}-p^{2} & -\frac{p(p-1)(2p-1)}{6}-\frac{p(p-1)}{2} \\
                       \end{array}
                     \right)\\
     &\equiv&\hspace{-2mm}\left(
                            \begin{array}{cc}
                              1-p & -p \\
                              p & p+1 \\
                            \end{array}
                          \right)\\
     &=&\hspace{-2mm}\left(
                   \begin{array}{cc}
                     1 & 0 \\
                     0 & 1 \\
                   \end{array}
                 \right)
                 +\left(
                            \begin{array}{cc}
                              -p & -p \\
                              p& p\\
                            \end{array}
                          \right)\pmod{p^{2}}.
\end{eqnarray*}
Combining with Eq.(\ref{identity}), we get $e=1$.
\end{proof}

\section{Regular embeddings of $K_{3[n]}$}

In this section, we assume that $m=3$ and $n=3^{e}k\ge 2$ where  $3\nmid k$.
Set $\Gamma=K_{3[n]}$, with the vertex set
\begin{equation*}
V(\Gamma)=\Delta_{1}\bigcup\Delta_{2}\bigcup\Delta_{3}~~\mbox{where}~~
\Delta_{i}=\{\gamma_{i1},\gamma_{i2},\cdots,\gamma_{in}\}
\end{equation*}
and two vertices are adjacent if and only if they are in different $\Delta_{i}$. Let $\MM$ be an orientably-regular embedding of $\Gamma$ and $G=\Aut^{+}(\MM)=\lg a, b\rg $ where
$\langle a\rangle =G_{\gamma_{11}}$ and $(\gamma_{11},\gamma_{21})^{b}=(\gamma_{21},\gamma_{11})$. Set $a^{2}=x$, $y=x^{b}$ and $H=\langle x,y\rangle$. Then $\Aut_{0}^{+}(\MM)=H$. By Proposition~\ref{rth}, $H=Q\times K$ where $Q$ is a $3$-group and $K$ is an abelian $3'$- group.  Let $P$ be a Sylow 3-subgroup of $G$ and we divide the discussions into two subsections according to $\Exp(P)=3^{e+1}$ or $\Exp(P)=3^{e}$.

\subsection{$\Exp(P)=3^{e+1}$}

\begin{theorem}\label{m33} Suppose  that $\Exp(P)=3^{e+1}$. Then
 $\MM\cong   \MM_3(k,e,j)$ where $j\in \Z_{2k\cdot 3^e}^*$ and   $\MM_3(k,e,j_1)\cong \MM_3(k,e,j_2)$  if and only if $j_1\equiv j_2\pmod{2\cdot 3^e}$.
 Moreover, all the maps $\MM_3(k,e,j)$ are chiral regular embeddings of $K_{3[n]}$ with the type $\{3^{e+1}, 2\cdot3^{e}k\}$ and the number of such maps is $2\cdot 3^{e-1}$ up to isomorphism.
\end{theorem}

\begin{proof} Recall that
\begin{eqnarray*}
  G_{3}(k,e)\hspace{-2mm}&=&\hspace{-2mm}\langle a,b\mid a^{2\cdot3^{e}k}=b^{2}=1, c=a^{3^{e}}b, a^{2\cdot3^{e}}=x_{1}, x_{1}^{b}=y_{1}, [x_{1},y_{1}]=1,c^{3^{e+1}}=1, \\
  &&\hspace{10mm} y_{1}^{a}=x_{1}^{-1}y_{1}^{-1},
       c^{a}=c^{2}x_{1}^{u}y_{1}^{\frac{1-3^{e}}{2}u}\rangle,
\end{eqnarray*}
where $u3^e\equiv 1\pmod k$ if $k>1$ (note that $c^{a}=c^{2}$ if $k=1$). Now we divide the proof into three steps.

\vskip 3mm
(1) Show that $G\cong G_{3}(k,e)$.
\vskip 2mm

Let $x_{1}=a^{2\cdot3^{e}}$,  $y_{1}=x_{1}^{b}$ and  $c=a^{3^{e}}b$.  Then
$K=\langle x_{1}, y_{1}\rangle \cong \Z_{k}\times\Z_{k}$.

Set $\widetilde{G}=G/Q$ and  let  $\widetilde{\M}$ be the quotient map of $\M$ induced by  $Q$. Then $\Aut^{+}(\widetilde{\M})\cong \widetilde{G}$ and $\widetilde{\M}$ is an orientably-regular embedding of $K_{3[k]}$. By the classification of orientably-regular embeddings of $K_{3[k]}$ for $k$ coprime to 3 (see \cite[Lemma 5.2]{ZD}), we have
\begin{equation}\label{q1}
\widetilde{G}=\langle \widetilde{a}, \widetilde{b}\mid \widetilde{a}^{2k}=\widetilde{b}^{2}=(\widetilde{a}\widetilde{b})^3
=\widetilde{1}, \widetilde{a}^{2}=\widetilde{x}, \widetilde{x}^{\widetilde{b}}=\widetilde{y}, [\widetilde{x},\widetilde{y}]=\widetilde{1}, \widetilde{y}^{\widetilde{a}}=
\widetilde{x}^{-1}\widetilde{y}^{-1}\rangle.
\end{equation}
Thus we get  $\widetilde{y_{1}}^{\widetilde{a}}=
\widetilde{(y^{3^{e}})}^{\widetilde{a}}=
\widetilde{x^{-3^{e}}}\widetilde{y^{-3^{e}}}=
\widetilde{x}_{1}^{-1}\widetilde{y}_{1}^{-1}$ and then $y_1^ay_1x_1\in Q\cap K=1$, that is $y_1^a=x_1^{-1}y_1^{-1}$.

Set $\og=G/K$ and  let     $\overline{\MM}$  be the quotient map of $\M$ induced by  $K$. Then  $\Aut(\overline{\MM})=\og .$
Noting that  Theorem~\ref{m1} holds for $p=3$, we have
\begin{equation}\label{bar}
\og '=\lg \olc\rg\quad\mbox{and}\quad\overline{c}^{\overline{a}}=\overline{c}^{2}.
\end{equation}

Now  we  show two facts:

\vskip 3mm
{\it  Fact 1:} $\lg K, c\rg=K\rtimes\langle c\rangle\unlhd G$ with
$|\langle c\rangle|=3^{e+1},~x_1^c=y_1,~y_1^c=x_1^{-1}y_1^{-1}~\mbox{and}~[c^3, x_1]=[c^3, y_1]=1$.

\vskip 2mm
Since $\lg \olc\rg=\overline{G}' \unlhd \og$ and $K\unlhd G$, we have $\lg K, c\rg\unlhd G$.
Write $g=ab$. By Eq.(\ref{q1}), we get $\widetilde{g}^{3}=1$.
Noting that $\langle\widetilde{a}^{2}\rangle=\langle \widetilde{x}\rangle=\langle\widetilde{x_{1}}\rangle$, we get $\widetilde{a}^{3^{e}-1}\in\langle \widetilde{x_{1}}\rangle$. Let $\widetilde{a}^{3^{e}-1}=\widetilde{x_{1}}^{i}$ for some integer $i$. Then $\tilde{c}=\widetilde{a^{3^{e}}b}=\widetilde{a}^{3^{e}-1}
\widetilde{a}\widetilde{b}=\widetilde{x_{1}}^{i}\widetilde{g}$. Since
 $\widetilde{x_{1}}^{\widetilde{g}}=
\widetilde{x_{1}}^{\widetilde{a}\widetilde{b}}=
\widetilde{x_{1}}^{\widetilde{b}}=\widetilde{y_{1}}$ and $\widetilde{y_{1}}^{\widetilde{g}}=
\widetilde{y_{1}}^{\widetilde{a}\widetilde{b}}=
(\widetilde{x_{1}}^{-1}\widetilde{y_{1}}^{-1})^{\widetilde{b}}
=\widetilde{x_{1}}^{-1}\widetilde{y_{1}}^{-1}$,
we have
\begin{equation*}
(\widetilde{c})^{3} = \widetilde{g}^{3}(\widetilde{x_{1}}^{i})^{\widetilde{g}^{3}} (\widetilde{x_{1}}^{i})^{\widetilde{g}^{2}} (\widetilde{x_{1}}^{i})^{\widetilde{g}}
=\widetilde{x_{1}}^{i}\widetilde{x_{1}}^{-i}\widetilde{y_{1}}^{-i} \widetilde{y_{1}}^{i}=1
\end{equation*}
and hence $c^{3}\in Q$. Then $|\langle c\rangle|=3^{e+1}$, since $\Exp(Q)=3^{e}$ and the order of $\overline{c}$ is $3^{e+1}$ in $\og$.
Noting that $y_{1}^{a^{i}}=x_{1}^{-1}y_{1}^{-1}$ for odd $i$ and  $y_{1}$ for even $i$,  we get
\begin{equation*}
x_{1}^{c}=x_{1}^{a^{3^{e}}b}=x_{1}^{b}=y_{1}, \, y_{1}^{c}=y_{1}^{a^{3^{e}}b}=(x_{1}^{-1}y_{1}^{-1})^{b}
=x_{1}^{-1}y_{1}^{-1}~\mbox{and}~[c^3, x_1]=[c^3, y_1]=1.
\end{equation*}
Since $\gcd\{|K|,|\langle c\rangle|\}=1$ and $K\unlhd G$,  we have $\langle K,c\rangle=\langle K\rangle\rtimes\langle c\rangle$.

\vskip 3mm
{\it  Fact 2:} $G=\lg K, c\rg  \lg a\rg$ with $c^{a}=c^{2}x^{u}y^{\frac{1-3^{e}}{2}u}$ where $u 3^{e}\equiv 1\pmod k$.

\vskip 3mm
Since $G=\langle a,b\rangle=\langle a,c\rangle$, we have $G=\lg K, c\rg  \lg a\rg$.
By Eq.(\ref{bar}), we can set $c^{a}=c^{2}x_{1}^{u}y_{1}^{v}$ for some integers $u$ and $v$ for $k\geq2$. The remaining  is to show $u 3^{e}\equiv 1\pmod k$ and $v\equiv\frac{1-3^{e}}{2}u\pmod k$.

One can deduce the following formulas  by induction on $i$:
\begin{equation}\label{10}
(c^{2}x_{1}^{u}y_{1}^{v})^{i}=\left\{
  \begin{array}{ll}
    c^{2i}, &~i\equiv0\pmod 3 \\
    c^{2i}x_{1}^{u}y_{1}^{v}, &~i\equiv1\pmod 3 \\
    c^{2i}x_{1}^{v}y_{1}^{v-u}, &~i\equiv2\pmod 3
  \end{array}
\right.
\end{equation}
and
\begin{equation}\label{11}
    c^{a^{2i}}=c^{4^{i}}x_{1}^{iu}y_{1}^{-iu}.
\end{equation}
Then we have
\begin{equation*}
c^{a^{2\cdot 3^{e}}}=c^{4^{3^{e}}}x_{1}^{3^{e}u}y_{1}^{-3^{e}u}
=cx_{1}^{3^{e}u}y_{1}^{-3^{e}u}.
\end{equation*}
On the other hand, $c^{a^{2\cdot 3^e}}=c^{x_{1}}=c(x_{1}^{-1})^{c}x_{1}=cx_{1}y_{1}^{-1}$.
Therefore $u 3^{e}\equiv 1\pmod k$.

\vskip 3mm
By Eq.(\ref{11}), we have
\begin{equation}\label{12}
c^{a^{3^{e}-1}}=c^{2^{3^{e}-1}}x_{1}^{\frac{3^{e}-1}{2}u}y_{1}^{\frac{1-3^{e}}{2}u}.
\end{equation}
Note that
$2^{3^{e}}\equiv 2\pmod 3$ and $2^{3^{e}-1}\equiv 1\pmod 3$.
Then by Eq.(\ref{12})~and~Eq.(\ref{10}), we have
\begin{equation}\label{13}
c^{a^{3^{e}}}=
(c^{2^{3^{e}-1}}x_{1}^{\frac{3^{e}-1}{2}u}y_{1}^{\frac{1-3^{e}}{2}u})^{a}
=(c^{2}x_{1}^{u}y_{1}^{v})^{2^{3^{e}-1}}
x_{1}^{(3^{e}-1)u}y_{1}^{\frac{3^{e}-1}{2}u}
=c^{2^{3^{e}}}x_{1}y_{1}^{\frac{3^{e}-1}{2}u+v}
\end{equation}
and
\begin{equation}\label{14}
(c^{2^{3^{e}}}x_{1}y_{1}^{\frac{3^{e}-1}{2}u+v})^{2^{3^{e}}}=
c^{2^{2\cdot3^{e}}}x_{1}^{\frac{3^{e}-1}{2}u+v}
y_{1}^{-\frac{3^{e}+1}{2}u+v}
=cx_{1}^{\frac{3^{e}-1}{2}u+v}y_{1}^{-\frac{3^{e}+1}{2}u+v}.
\end{equation}
Since $c=a^{3^{e}}b$, we get $c^{b}=c^{a^{3^{e}}}$.
Then by Eq.(\ref{13}) and Eq.(\ref{14}), we have
 \begin{equation*}
c=c^{c}=c^{a^{3^{e}}b}=
(c^{2^{3^{e}}}x_{1}y_{1}^{\frac{3^{e}-1}{2}u+v})^{b}
=(c^{b})^{2^{3^{e}}}x_{1}^{\frac{3^{e}-1}{2}u+v}y_{1}
=cx_{1}^{(3^{e}-1)u+2v}y_{1}^{\frac{3^{e}-1}{2}u+v}.
\end{equation*}
Therefore $x_{1}^{(3^{e}-1)u+2v}y_{1}^{\frac{3^{e}-1}{2}u+v}=1$ and hence $v\equiv\frac{1-3^{e}}{2}u~\pmod k$.

\vskip 3mm
From the \emph{Fact 1} and \emph{Fact 2}, $G$ satisfies all the relations of $G_3(k,e)$ and then $G$  is an homomorphic image  of $G_3(k,e)$. A checking shows that $|G_3(k,e)|=|G|$.
Therefore $G\cong G_3(k,e)$.

\vskip 3mm (2) Determination of  $\M$.

\vskip 3mm Recall that $\MM_3(k,e,j)=\M\big(G_{3}(k,e);a^{j},b\big)$ where $j\in \Z_{2n}^{*}$. Since $G\cong G_{3}(k,e)$, we have $\M$ is isomorphic to one of the maps $\M_{3}(k,e,j)$.
If $\MM_3(k,e,j_{1})\cong\MM_3(k,e,j_{2})$ for two parameters
$j_{1},j_{2}\in\Z_{2n}^{*}$, then there exists an automorphism $\psi$ of $\Aut\big(G_{3}(k,e)\big)$ such that $\psi(a^{j_{1}})=a^{j_{2}}$ and $\psi(b)=b$. Clearly, $\psi$ induces an automorphism $\overline{\psi}$ of $\overline{G}$ such that $\overline{\psi}(\overline{a}^{j_{1}})=\overline{a}^{j_{2}}$ and $\overline{\psi}(\overline{b})=\overline{b}$. By  Proposition \ref{m1}, we get
$j_{1}\equiv j_{2}\pmod{2\cdot3^{e}}$.

Conversely, suppose  that $j_{1}\equiv j_{2}~\pmod{2\cdot3^{e}}$ for $j_1, j_2\in \Z_{2k\cdot 3^e}$.  Then $j=j_2j_1^{-1}\equiv 1\pmod{2\cdot 3^e}$.
It suffices to  show that the mapping $\psi:~a\mapsto a^{j},~b\mapsto b$ can be extended to an automorphism of $G_3(k,e)$. That is, $\psi$ can be extended to a bijection
preserving all the defining  relations of $G_3(k,e)$.
By the presentation of $G_3(k,e)$, we can set
\begin{equation*}
\psi(x_1)=x_1^j,\, \psi(y_1)=y_1^j\quad\mbox{and}\quad\psi(c)=(a^{j})^{3^e}b.
\end{equation*}
Then the four formulas $$\psi(a)^{2\cdot3^{e}k}=\psi(b)^{2}=1, \, [\psi(x_{1}),\, \psi(y_{1})]=1,\,
 \psi(x_1)^{\psi(a)}=\psi(x_1)\,
  \psi(y_1)^{\psi(a)}=\psi(x_1)^{-1}\psi(y_1)^{-1}$$
 clearly hold.
Since $$\psi(c)=(a^{j})^{3^e}b=a^{2\cdot 3^e\frac{j-1}2}a^{3^e}b=x_1^{\frac {j-1}2}c,$$ $$(x_1^{\frac {j-1}2}c)^{3}=x_{1}^{\frac {j-1}2}c^{2}(x_{1}^{\frac {j-1}2})^{c}x_{1}^{\frac {j-1}2}c=x_{1}^{\frac {j-1}2}c^{3}(y_{1}^{\frac {j-1}2}x_{1}^{\frac {j-1}2})^{c}=x_{1}^{\frac {j-1}2}c^{3}x_{1}^{-\frac {j-1}2}=c^{3},$$ we know
that $\psi(c)$ and $c$ have the same order.
Set $j=2\cdot3^{e}i+1$.
Then $a^{j}=x_{1}^{i}a$ and hence
\begin{equation*}
c^{a^{j}}=c^{x_{1}^{i}a}=\big(c(x_{1}^{-i})^{c}x_{1}^{i}\big)^{a}
=(cy_{1}^{-i}x_{1}^{i})^{a}=c^{2}x_{1}^{u}y_{1}^{\frac{1-3^{e}}{2}u}x_{1}^{2i}y_{1}^{i}
=c^{2}x_{1}^{u+2i}y_{1}^{\frac{1-3^{e}}{2}u+i}.
\end{equation*}
Noting that $\psi(c)^{2}=(x_{1}^{\frac {j-1}2}c)^{2}=x_{1}^{\frac {j-1}2}cy_{1}^{\frac {j-1}2}$, we have
\begin{equation}\label{15}
\psi(c)^{a^{j}}=(x_{1}^{\frac{j-1}{2}}c)^{a^{j}}=x_{1}^{\frac{j-1}{2}}c^{a^{j}}
=x_{1}^{\frac{j-1}{2}}c^{2}x_{1}^{u+2i}y_{1}^{\frac{1-3^{e}}{2}u+i}
=\psi(c)^{2}x_{1}^{u+2i}y_{1}^{\frac{1-3^{e}}{2}u-\frac{j-1}{2}+i}.
\end{equation}
Since $j=2\cdot3^{e}i+1$, $3^{e}u\equiv1~\pmod k$,
we have
\begin{equation*}
u+2i\equiv u+2i\cdot3^{e}u\equiv(1+2i\cdot3^{e})u\equiv uj\pmod k,
\end{equation*}
and
\begin{equation*}
\frac{1-3^{e}}{2}u-\frac{j-1}{2}+i\equiv\frac{1-3^{e}-3^{e}j+3^{e}+
2\cdot3^{e}i}{2}u\equiv j\frac{1-3^{e}}{2}u\pmod k.
\end{equation*}
Then by Eq.(\ref{15}), we have $\psi(c)^{a^{j}}=\psi(c)^{2}(x_{1}^{j})^{u}(y_{1}^{j})^{\frac{1-3^{e}}{2}u}$.
Thus,  $\psi$ can be indeed  extended to a bijection
preserving all the defining  relations of $G_3(k,e)$.
\vskip 2mm
In summary,   $\M_{4}(k,e,j_{1})\cong\M_{4}(k,e,j_{2})$ if and only if $j_{1}\equiv j_{2}~\pmod{2\cdot3^{e}}$.

\vskip 3mm (3) Determination of  the number and type of the resulting maps.
 \vskip 3mm

By Lemma \ref{cos}, all of the maps $\M_{3}(k,e,j)$ for $j\in \Z_{2n}^{*}$ are orientably-regular embeddings of $K_{2[n]}$.
By (2), $\MM_3(k,e,j_{1})\cong\MM_3(k,e,j_{2})$ if
and only if $j_1\equiv j_2\pmod{2\cdot 3^e}$. Noting that for $j\in \Z_{2n}^{*}$ with $(j, 2\cdot 3^e)=1$,   the set $\{ j+i \cdot 2\cdot 3^e\di 0\le i\le k-1 \}$ contains at least one number which is coprime to $2\cdot 3^ek$,  then we have that the number of maps in this family is $\phi(2\cdot 3^e)=2\cdot 3^{e-1}$.
With the same arguments  as in  \emph{Fact 1} of (1),
one may check $|\langle a^jb\rangle|=3^{e+1}$. Thus, all of the resulting maps have type $\{3^{e+1}, 2\cdot3^{e}k\}$. Clearly, they are all chiral.
\end{proof}

\subsection{$\Exp(P)=3^e$}
 The following theorem quoted from   \cite[Lemma 5.2]{ZD} gives a determination of the case when $H$ is abelian.
\begin{theorem} \label{abelian2} If $H$ is abelian, then
  \begin{equation*}
  \M\cong\M_3(k,e,0,l,j)
  \end{equation*}
where $(l,j)=(0,1)$ if $e=0$ and  $(l,j)=(0, 1)$,~$(1, 1)$~or~$(1, -1)$ if $e\geq1$.
The resulting maps $\M_3(k,e,0,0,1)$ and $\M_3(k,e,0,1,\pm1)$ have type $\{3, 2n\}$ and $\{9, 2n\}$ respectively, and all of them are orientably-regular embeddings of $K_{3[n]}$.
\end{theorem}

If $H$ is nonabelian, then we have the following theorem.
\begin{theorem}\label{m3} Suppose that $H$ is nonabelian and $\Exp(P)=3^e$.
Then  $\MM$ is isomorphic to one of the maps $\MM_4(k, e,  1, l, j)$, where $l=0, \pm 1$ and $j=\pm 1$.
Moreover, all of the maps $\MM_4(k, e,  1, l, j)$ are chiral regular embeddings of $K_{3[n]}$ with the type $\{3, 2n\}$ if $l=0$ and $\{9, 2n)\}$ if $l=\pm 1$.
\end{theorem}
\begin{proof}
We divide the proof into two steps.

\vskip 3mm (1)~~Show that $G$ is isomorphic to one of the following groups
\begin{eqnarray*}
   G_{4}(k,e,1,l)\hspace{-2mm}&=&\hspace{-2mm}\langle a, b\mid a^{2n}=b^2=1, a^2=x, x^b=y, [x,y]=x^{\frac {n}3}y^{-\frac {n}3}, y^{a}=x^{-1}y^{-1},\\
    &&\hspace{9mm}(ab)^3=x^{\frac {ln}3}y^{-\frac {ln}3}\rangle,
 \end{eqnarray*}
 where~$l=0,\pm1$.

Set $\og=G/K$ and let $\overline{\MM}$  be the quotient map of $\M$ induced by $K$. Then  $\M$ is an orientably-regular embedding of $K_{3[3^{e}]}$ and $\Aut(\overline{\MM})=\og$. Clearly,  $\Exp(\overline{P})=\Exp(P)=3^e$ and $\overline{H}\cong Q$. By Theorem \ref{abelian}, we have that $\mathrm{Z}(Q)=\langle g^{3}\mid g\in Q\rangle\cong\Z_{3^{e-1}}\times\Z_{3^{e-1}}$ and $Q'\cong\Z_{3}$. Since $H=Q\times K$ and $K$ is abelian, we have $\mathrm{Z}(H)=\langle x^{3},y^{3}\rangle$ and $H'=\langle[x,y]\rangle\cong\Z_{3}$.
Set $[x,y]=x^{\frac{n}{3}i}y^{\frac{n}{3}j}$ where $i,j=0\;\mbox{or}\;\pm1$. Then from
  $$y^{-\frac{n}{3}j}x^{-\frac{n}{3}i}=[x,y^{-1}]=[y,x]=[x,y]^{b}=(x^{\frac{n}{3}i}y^{\frac{n}{3}j})^{b}=y^{\frac{n}{3}i}x^{\frac{n}{3}j},$$
 we get  $i=-j$ and hence $[x,y]=x^{\pm\frac{n}{3}}y^{\mp\frac{n}{3}}$. Noting that these two cases can be  interconvertible by replacing $a$ by $a^{-1}$, we can set $[x,y]=x^{\frac{n}{3}}y^{-\frac{n}{3}}$ without loss of any generalities. Therefore,
 $H$ has a presentation
\begin{equation*}
 H=\langle x,y\mid x^{n}=y^{n}=[x^{3},y]=[x,y^{3}]=1,~[x,y]=x^{\frac{n}{3}}y^{-\frac{n}{3}}\rangle.
\end{equation*}
Set $[x,y]=w$. One can check that
the multiplications and powers in $H$ are given by
\begin{equation}\label{chf}
(x^{i}y^{l})(x^{r}y^{d})=x^{i+r}y^{l+d}w^{-ld},\quad
(x^{i}y^{l})^{r}=x^{ri}y^{rl}w^{-\frac{r(r-1)}{2}il}.
\end{equation}

Since $G$ is an extension of $H$ by $G/H\cong S_{3}$ with  $a^{2}=x$, $b^{2}=1$ and $x^{b}=y$, it follows that  $G$ can be determined by the following relations
\begin{equation*}
y^a=x^uy^v,~~(ab)^3=x^{s}y^{t},
\end{equation*}
where  $u$, $v$, $s$ and $t$ are undetermined parameters. Set $c=ab$. Then
\begin{equation*}
x^{c}=x^{b}=y,~~ x^{c^{2}}=y^{c}=(x^{u}y^{v})^{b}=y^{u}x^{v}~\mbox{and}~
x^{c^{3}}=(y^{u}x^{v})^{c}=(y^{u}x^{v})^{u}y^{v}.
\end{equation*}
By Eq.(\ref{chf}), we have
\begin{equation*}
(y^{u}x^{v})^{u}=(x^{v}y^{u}w^{-uv})^{u}
=x^{uv}y^{u^{2}}w^{-\frac{u(u-1)}{2}uv-u^{2}v}=x^{uv}y^{u^{2}}w^{-\frac{u(u+1)}{2}uv}
\end{equation*}
and hence
\begin{equation*}
    x^{c^{3}}=x^{uv}y^{u^{2}+v}w^{-\frac{u(u+1)}{2}uv}
    =x^{uv-\frac{nu(u+1)}{6}uv}y^{u^{2}+v+\frac{nu(u+1)}{6}uv}.
\end{equation*}
On the other hand, since $\Exp(P)=3^{e}$, we know that 3 divides both $s$ and $t$. It follows that $c^3\in \mathrm{Z}(H)$ and hence $x^{c^{3}}=x$. Therefore
\begin{equation}\label{eqar1}
\left\{\begin{gathered}
    uv-\frac{nu(u+1)}{6}uv\equiv1\pmod n,\\
   u^{2}+v+\frac{nu(u+1)}{6}uv\equiv0\pmod n.
    \end{gathered}\right.
\end{equation}
Since
\begin{equation*}
y^{x}=y^{a^{2}}=(x^{u}y^{v})^{a}
=x^{u}(x^{u}y^{v})^{v}=x^{u+uv}y^{v^{2}}w^{-\frac{v(v-1)}{2}uv}
=x^{u+uv-\frac{nv(v-1)}{6}uv}y^{v^{2}+\frac{nv(v-1)}{6}uv}
\end{equation*}
and
\begin{equation*}
y^{x}=[x,y^{-1}]y=yw^{-1}=x^{-\frac{n}{3}}y^{1+\frac{n}{3}},
\end{equation*}
we have
\begin{equation}\label{eqar2}
\left\{\begin{gathered}
    u+uv-\frac{nv(v-1)}{6}uv\equiv-\frac{n}{3}\pmod n,\\
  v^{2}+\frac{nv(v-1)}{6}uv\equiv1+\frac{n}{3}\pmod n.
    \end{gathered}\right.
\end{equation}
By solving the equations
(\ref{eqar1})~and~(\ref{eqar2}), one can obtain $u\equiv v\equiv-1\pmod n$, that is   $y^{a}=x^{-1}y^{-1}$.
Since $c^{3}=(ab)^{3}=x^{s}y^{t}$ and $3$ divides  both $s$ and $t$, we have
\begin{equation*}
 x^{s}y^{t}=(x^{s}y^{t})^{ab}=[x^{s}(x^{-1}y^{-1})^{t}]^{b}=(x^{s-t}y^{-t})^{b}=x^{-t}y^{s-t}
\end{equation*}
and hence
\begin{equation*}
     s\equiv  -t(\mod{n}),\,
         t\equiv s-t(\mod n),
       \end{equation*}
that is  $3s\equiv3t\equiv0\pmod{n}$.
It follows that $(ab)^3=x^{-t}y^{t}=x^{\frac{ln}{3}}y^{-\frac{ln}{3}}$,
where $l=0~\mbox{or}~\pm1$.
Now we have proved that $G$ satisfies all the defining relations of $G_{4}(k,e,1,l)$. Checking directly,  one has $|G_{4}(k,e,1,l)|=|G|$.  Therefore $G\cong G_{4}(k,e,1,l)$.

\vskip 3mm (2) Determination  the map $\M$.

Take two parameters $j_{1},\,j_{2}\in\Z_{2n}^{*}$. Then one may verify that the mapping $a^{j_{1}}\mapsto a^{j_{2}}, \, b\mapsto b$ can be extended to an automorphism of $G_{4}(k,e,1,l)$ if and only if
~$j_{1}\equiv j_{1}\pmod 3$. Therefore, for given $k,e,l$,  $\M$ is isomorphic to one of the maps $\M_{4}(k,e, 1, l,j)$ where $j=\pm1$.

The remaining is to show  different $l$ give nonisomorphic  maps.
Suppose that for $l_1\ne l_2$, $\M_{4}(k,e, 1, l_1,j_1)$ is  isomorphic to $\M_{4}(k,e, 1, l_2, j_2)$.
  Then there exists an isomorphism $\psi$ from $G_4(k,e,1,l_1)$ to $G_4(k,e,1,l_2)$.
Set
\begin{eqnarray*}
   G_{4}(k,e,   ,l_1)&=&\hspace{-2mm}\langle a', b'\mid a'^{3^{e}}=b'^2=1, a'^2=x', x'^b=y', [x',y']=x'^{3^{e-1}}y'^{-3^{e-1}}, \\
    &&\hspace{10mm}y'^{a}=x'^{-1}y'^{-1},(a'b')^3=x'^{l_1\frac {n}3}y'^{-\l_1\frac {n}3}\rangle,
  \end{eqnarray*}
and
 \begin{eqnarray*}
   G_{4}(k,e,1,l_2)\hspace{-3mm}&=&\hspace{-2mm}\langle a, b\mid a^{3^{e}}=b^2=1, a^2=x, x^b=y, [x,y]=x^{3^{e-1}}y^{-3^{e-1}},\\
    &&\hspace{8mm} y^{a}=x^{-1}y^{-1},(ab)^3=x^{l_2\frac {n}3}y^{-l_2\frac {n}3}\rangle.
 \end{eqnarray*}
  Since for any given $k,e, l$, we have two maps $\M_{4}(k,e, 1, l,j)$ where $j=\pm 1$,
  we may assume that $\psi (a')=a^{j}$ where $j=\pm 1$ and $\psi(b')=b$.
 Then  from $(\psi(a')\phi(b'))^3=\phi(x')^{l_1\frac {n}3}\psi(y')^{-l_1\frac {n}3}$, we get
$(a^jb)^3=x^{jl_1\frac {n}3}y^{-jl_1\frac {n}3}$. However, one can check that this equation does not hold provided $l_1\ne l_2$.

By Lemma~\ref{cos}, all of the maps $\M_{4}(k,e, 1, l,j)$ are orientably-regular embeddings of $K_{3[n]}$. Clearly, all of these maps are chiral.
By a simple calculation,
we have that the order of $a^{\pm 1}b$ is $3$ if $l=0$ and $9$ if $l=\pm 1$.
Therefore the resulting maps have type $\{3, 2n\}$ if $l=0$ and $\{9, 2n\}$ if $l=\pm 1$.
\end{proof}

\section{Proof of Theorem~\ref{main}}
Let $\M$ be an orientably-regular embedding of $\Gamma=K_{m[n]}$ where $m\ge 3$ and $n\ge 2$, and let $\Aut^{+}_{0}(\M)$ be the normal subgroup  of $\Aut^{+}(\M)$ consisting of automorphisms preserving each part setwise. By Proposition~\ref{rth}, either $\Gamma=K_{p[p^e]}$ where $p\ge 5$ is prime, or  $\Gamma=K_{3[n]}$. For the case $\Gamma=K_{p[p^e]}$ where $p\ge 5$, we get $\MM\cong\MM_{1}(p,e,j)$ or $\MM_{2}(p,j)$ in Section 3. For the case $\Gamma=K_{3[n]}$,  we get $\MM \cong\M_{3}(k,e,j)$ or $\M_{4}(k,e, i, l,j)$ in Section 4. As proved in Section 3 and Section 4, all of the resulting maps are indeed the orientably-regular embedding of $\Gamma=K_{m[n]}$ and such maps are unique determined by the given parameters. Finally,   Table 1 and Table 2 directly can be obtained from Theorem \ref{m1}, Theorem \ref{abelian}, Theorem \ref{m33} and Theorem \ref{m3}. This finish the proof of Theorem~\ref{main}.\qed


\begin{thebibliography}{99}
\addtolength{\itemsep}{-3mm}
\footnotesize
\bibitem{Big1}
N. L. Biggs, Classification of complete maps on orientable surfaces, \emph{Rend
Math} \textbf{4(6)}(1971), 132--138.

\bibitem{Big2}
N. L. Biggs, Algebraic Graph Theory, second ed., {\it Cambridge
University Press, Cambridge}, 1993.

    \bibitem{CM}
H. S .M. Coxeter, W.O.J. Moser, Generators and Relations for Discrete Groups, fourth ed., Springer, Berlin, 1984.

\bibitem{Conder1}
M.D.E. Conder, Regular maps and hypermaps of Euler characteristic $-1$ to $-200$,
\emph{J. Combin. Theory Ser. B} \textbf{99} (2009) 455-459, with associated lists of computational data available at http://www.math.auckland.ac.nz/~conder/hypermaps.html.

\bibitem{Conder}
M. Conder, I. M. Isaacs, Derived subgroups of products of an abelian and a cyclicsubgroup, \emph{J. London Math. Soc.} \textbf{69}(2004), 333--348.

%\bibitem{DM}
%J. D. Dixon and B. Mortimer, \emph{Permutation groups},  Springer, New York, 1996.

\bibitem{DJKNS1}
S. F. Du, G. A. Jones, J. H. Kwak, R. Nedela and M. $\rm\check{S}koviera$, Regular embeddings of $K_{n,n}$ where
n is a power of 2. I: Metacyclic case, \emph{European J. Combin.} \textbf{28}(2007) 1595--1609.
\bibitem{DJKNS2}
S. F. Du, G.A. Jones, J.H. Kwak, R. Nedela and M. $\rm\check{S}koviera$, Regular embeddings of $K_{n,n}$ where
n is a power of 2. II: The non-metacyclic case, \emph{European J. Combin.} \textbf{31}(2010) 1946--1956.
\bibitem{DKN}
S. F. Du, J.H. Kwak, R. Nedela, Regular maps with $pq$ vertices, \emph{J. Algebraic Combin.} \textbf{19}(2004), 123--141.
\bibitem{DKN1}
S. F. Du, J. H. Kwak, R. Nedela, Regular embeddings of complete
multipartite graphs, \emph{European J. Combin.} \textbf{26}(2005), 505--519.
\bibitem{GNSS}
A. Gardiner, R. Nedela, J.\v{S}ir\'a\v{n}, and M. \v{S}koviera, Characterization of graphs which underlie regular maps on closed surfaces, \emph{J Lond Math Soc} \textbf{59}(1999), 100--108.
\bibitem{H}
B. Huppert, \"Uber das Produkt von paarweise vertauschbaren zyklischen Gruppen, \emph{Math. Z.} \textbf{58}(1953), 243-264.
\bibitem{Hup}
B. Huppert, Endliche Gruppen I, Springer-Verlag, Berlin, 1967.
\bibitem{Ito}
N. Ito, \"Uber das Produkt von zwei abelschen Gruppen, \emph{Math. Z.} \textbf{62}(1955) 400--401.
\bibitem{JJ}
L. D. James and G. A. Jones, Regular orientable imbeddings of complete
graphs, \emph{J Combin Theory Ser B} \textbf{39}(1985), 353--367.
\bibitem{Jon1}
G. A. Jones, Regular embeddings of complete bipartite graphs: classification and
enumeration, \emph{Proc. London Math. Soc.} \textbf{101}(2010), 427--453.
\bibitem{JNS1}
G. A. Jones, R. Nedela and M. \v{S}koviera, Regular embeddings of $K_{n,n}$ where n is an odd prime power, \emph{European J. Combin.} \textbf{28}(2007), 1863--1875.
\bibitem{JNS2}
G. A. Jones, R. Nedela and M. \v{S}koviera, Complete bipartite graphs with a unique regular embedding, \emph{J. Combin. Theory Ser. B} \textbf{98}(2008), 241--248.
\bibitem{KK1}
J. H. Kwak and Y.S. Kwon, Regular orientable embeddings of complete bipartite graphs, \emph{J. Graph
Theory} \textbf{50}(2005), 105--122.
\bibitem{KK2}
J. H. Kwak and Y.S. Kwon, Classification of reflexible regular embeddings and self-Petrie dual regular embeddings of complete bipartite graphs, \emph{Discrete Math.} \textbf{308}(2008) 2156--2166.
\bibitem{KK3} J.~H.~Kwak and Y.~S.~Kwon, Classification of nonorientable regular embeddings of
complete bipartite graphs, \emph{J. Combin. Theory, Ser. B} \textbf{101}(2011) 191--205.
\bibitem{NS}
R. Nedela, M. \v{S}koviera, Exponents of orientable maps, \emph{Proc. London Math. Soc.} \textbf{75}(1997) 1--31.
\bibitem{NSZ}
R. Nedela, M. \v{S}koviera, and A. \rm Zlato\v{s}, Regular embeddings of complete bipartite graphs, \emph{Discrete Math.} \textbf{258}(2002) 379--381.
\bibitem{Wie} H. Wielandt, Uber das produkt von paarweise abelschen gruppen, \emph{Math. Z.} \textbf{62}(1955), 1--7.
    \bibitem{Wil} S. E. Wilson, Cantankerous maps and rotary
embeddings of $K_n$, \emph{J. Combin. Theory Ser. B}  \textbf{47}(1989), 262--273.
\bibitem{ZD} J. Y. Zhang, S. F. Du, On the Orientable Regular Embeddings of
Complete Multipartite Graphs, \emph{European J. Combin.}, \textbf{33}(2012) 1303--1312.
\end{thebibliography}
 \end{document}